\documentclass{amsart}

\usepackage{amssymb}
\usepackage{hyperref}
\usepackage{comment}

\usepackage{tikz}
\tikzset{node distance=2cm, auto}
\usepackage{longtable}

\usepackage[initials,alphabetic]{amsrefs}

\newtheorem{theorem}{Theorem}[section]
\newtheorem{proposition}[theorem]{Proposition}
\newtheorem{conjecture}[theorem]{Conjecture}

\theoremstyle{definition}
\newtheorem{definition}[theorem]{Definition}

\newcommand\cO{\mathcal{O}}

\renewcommand\AA{\mathbb{A}}
\newcommand\CC{\mathbb{C}}

\newcommand\PP{\mathbb{P}}
\newcommand\QQ{\mathbb{Q}}
\newcommand\RR{\mathbb{R}}
\newcommand\ZZ{\mathbb{Z}}

\newcommand\rH{\mathrm{H}}

\newcommand\rmd{\mathrm{d}}
\newcommand{\rmi}{\mathrm{i}}

\newcommand{\into}{\hookrightarrow}
\newcommand{\Ext}{\mathrm{Ext}}

\newcommand{\cExt}{\mathcal{E}xt}
\newcommand{\cHom}{\mathcal{H}om}

\DeclareMathOperator{\Spec}{Spec}
\DeclareMathOperator{\rank}{rank}

\def\conv#1{\mathrm{conv} \left\{ #1  \right\} } 

\def\vectortwo#1#2{\begin{pmatrix}
#1 \\ #2
\end{pmatrix}}

\def\vector#1#2#3{\begin{pmatrix}
#1 \\ #2 \\ #3
\end{pmatrix}}

\def\pow#1{[ \! [ #1 ] \! ] }
\def\Def{\mathit{Def}}
\newcommand{\dP}{\mathrm{dP}}

\title[An example of Mirror Symmetry for Fano threefolds]{An example of Mirror Symmetry \\  for Fano threefolds}

\author{Andrea Petracci}

\address{Institut f\"ur Mathematik, Freie Universit\"at Berlin, Arnimallee 3, 14195 Berlin, Germany}

\email{andrea.petracci@fu-berlin.de}

\begin{document}

\begin{abstract}
In this note we illustrate the Fanosearch programme of Coates, Corti, Galkin, Golyshev, and Kasprzyk in the example of the anticanonical cone over the smooth del Pezzo surface of degree 6.
\end{abstract}

\maketitle 

\section{Introduction}

\subsection{Aim}

The Fanosearch programme of Coates, Corti, Galkin, Golyshev, and Kasprzyk \cite{mirror_symmetry_and_fano_manifolds} studies Fano varieties via Mirror Symmetry. In this context it is crucial to study toric degenerations of smooth Fano varieties, or conversely smoothings of toric Fano varieties.
Toric Fano varieties are associated to certain lattice polytopes, called Fano polytopes; some combinatorial input on a Fano polytope conjecturally allows to construct a deformation of the corresponding toric Fano. This is also reflected by Mirror Symmetry, where the combinatorial input is encoded by certain special Laurent polynomials.
The goal of this note is to illustrate this programme in a specific example where two different combinatorial inputs on the same polytope produce two different smoothings of the same toric Fano threefold.

\subsection{The example}

The example we consider is the projective cone over the anticanonical embedding of the smooth del Pezzo surface of degree $6$. This threefold, denoted by $X$, is a toric Fano and has an isolated Gorenstein canonical non-terminal singularity at the vertex of the cone.
The deformations of this singularity have been studied by 
Altmann \cite{altmann_versal_deformation}; we will recall Altmann's results in \S\ref{sec:minkowski_deformations_U}.
In \S\ref{sec:two_smoothings_of_X} we will see that the base of the miniversal deformation (or equivalently the Kuranishi family) of the projective threefold $X$ has two irreducible components, which deform $X$ to two different smooth Fano threefolds, namely:
\begin{itemize}
\item a general element $X_2$ of the linear system $\vert \cO_{\PP^2 \times \PP^2} (1,1) \vert$,
\item $X_3 = \PP^1 \times \PP^1 \times \PP^1$.
\end{itemize}
(The reason for the subscripts $2$ and $3$ will be evident later.)
These two smooth Fanos are connected via deformation through $X$, but cannot be connected via a deformation with smooth fibres, as their Betti numbers are different.

As $X$ is toric, by means of toric geometry, we can associate to $X$ a 3-dimensional lattice polytope $P \subseteq \RR^3$ which is a hexagonal pyramid (see the precise definition in \eqref{eq:pyramid} and Figure~\ref{fig:piramide}). The hexagonal facet of $P$ is denoted by $F$ (see \eqref{eq:hexagon} and the left part of Figure~\ref{fig:Minkow}). In Proposition~\ref{prop:deformations_X} we will see that the two smoothings of $X$ are associated to some combinatorial additional data on the polytope $P$. More precisely, they correspond to the two maximal Minkowski decompositions of the hexagon $F$ (see \eqref{eq:minkowski_dec_segments} and \eqref{eq:minkowski_dec_triangles}, and Figure~\ref{fig:Minkow}). We will introduce the notion of Minkowski sum and Minkowski decomposition in \S\ref{sec:minkowski_deformations_U}.

\begin{figure}[t]
\centering
\includegraphics[width=5cm]{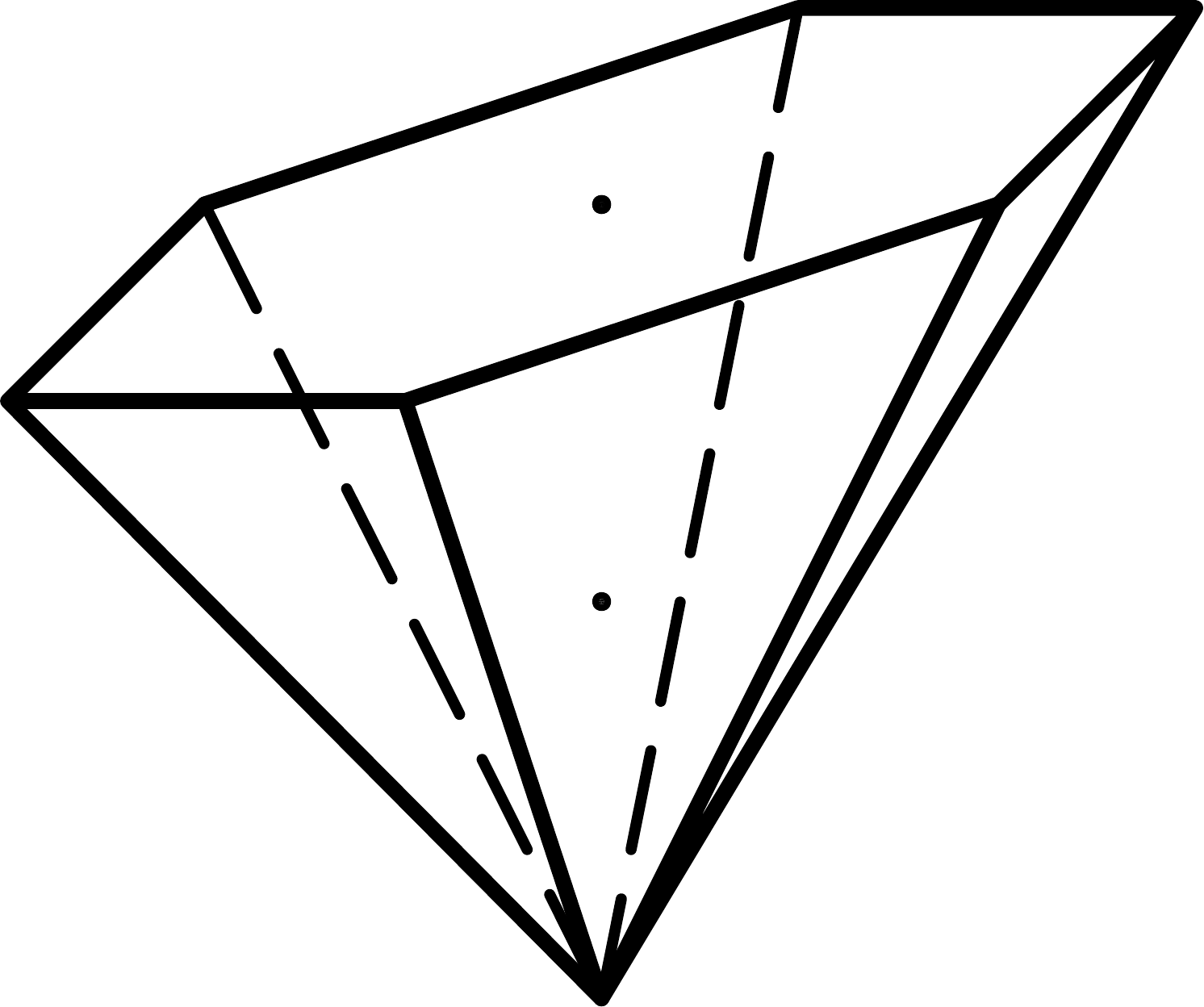}
\caption{The 3-dimensional lattice polytope $P$ associated to $X$.}
\label{fig:piramide}
\end{figure}

Now we consider the Laurent polynomials in 3 variables which are supported on $P$, i.e. Laurent polynomials $f \in \CC [ x^\pm, y^\pm, z^\pm]$ such that if the monomial $x^i y^j z^k$ appears in $f$ then the point $(i,j,k) \in \ZZ^3$ lies in $P$. Among these Laurent polynomials, we consider those which have coefficient 1 on the vertices of $P$ and have coefficient 0 on the origin of $\RR^3$; this gives rise to the following 1-dimensional family:
\begin{equation*}
f_a = z \left(a + x+xy+y+\frac{1}{x} + \frac{1}{xy} + \frac{1}{y} \right) + \frac{1}{z}
\end{equation*}
with parameter $a \in \CC$.

One can show that $f_2$ is mirror to $X_2$ and $f_3$ is mirror to $X_3$. In other words, a certain generating function for some Gromov--Witten invariants of $X_2$, called \emph{quantum period} (see \S\ref{sec:quantum_periods}), is equal to a certain power series, called \emph{classical period} (see \S\ref{sec:classical_periods}), associated to $f_2$, and the same holds for $X_3$ and $f_3$. Here we are using the formulation of the Mirror Symmetry correspondence between Fanos and Landau--Ginzburg models that is given in \cite{mirror_symmetry_and_fano_manifolds,
przyjalkowski_landau_ginzburg_fano, victor_lg} and summarised in \S\ref{sec:Fano-Landau_Ginzburg}.

We will see that the Laurent polynomial $f_2$ is closely related to the combinatorial input given by the Minkowski decomposition of the hexagon $F$ which is associated to the smoothing of $X$ to $X_2$. Analogously, $f_3$ is closely related to the Minkowski decomposition of the hexagon $F$ which is associated to the smoothing of $X$ to $X_3$.

\subsection{The general picture} \label{sec:general}

What we have described in the case of the projective cone over the del Pezzo surface of degree 6 is an instance of the following conjecture, which is still slightly vague.

\begin{conjecture}[{\cite{mirror_symmetry_and_fano_manifolds}}] \label{conj}
Let $Q$ be a Fano polytope of dimension $3$ and let $X_Q$ be the corresponding toric Fano threefold. Assume that $X_Q$ has Gorenstein singularities.
From some ``combinatorial input'' on $Q$ one constructs 
\begin{itemize}
\item[(i)] \ \ a smoothing $V$ of $X_Q$ and
\item[(ii)]  \ \ a Laurent polynomial $f$ supported on $Q$
\end{itemize}
such that $f$ is mirror to $V$.
\end{conjecture}

The definition of ``mirror'' that we are using
comes from \cite{mirror_symmetry_and_fano_manifolds,
przyjalkowski_landau_ginzburg_fano, victor_lg} and is given in Definition~\ref{def:mirror}.

If the toric variety $X_Q$ is smooth (there are 18 cases), then the polytope $Q$ has only triangular facets which are standard simplices and $X_Q$ is rigid. Thus $V = X_Q$ and $f$ is uniquely determined by insisting that it has coefficient 1 on vertices of $Q$ and coefficient 0 on the origin. This case was already known by Givental \cite{givental_toric_ci, givental_equivariant} who proved that $f$ is mirror to $X_Q$.

In the example considered in this note, the combinatorial input on $P$ is the choice of a maximal Minkowski decomposition of the facet $F$ of $P$. There are two such choices which lead to two different smoothings of $X$ and to two different Laurent polynomials.

An interesting case, which is not too restrictive, is the following: the combinatorial input is the choice of a Minkowski decomposition of each facet of $Q$ into $A$-triangles. Here an $A$-triangle is either a unitary segment or a lattice triangle which is $\ZZ^2 \rtimes \mathrm{GL}_2(\ZZ)$-equivalent to the convex hull of the points $(0,0), (0,1), (\ell,0)$, for some integer $\ell \geq 1$.
For example, both maximal Minkowski decompositions of the hexagon $F$  are decompositions into $A$-triangles. In these circumstances one can easily construct a Laurent polynomial $f$ which is supported on $Q$ and depends on the choice of the Minkowski decompositions of the facets of $Q$ (see \cite{sigma}, where such Laurent polynomial $f$ is called a \emph{Minkowski polynomial}). 
In joint work with Corti and Hacking \cite{chp}, we construct a smoothing $V$ of $X_Q$, under a slight additional assumption which is necessary by \cite{petracci_local_to_global_obstruction}. It is conjectured that $f$ is mirror to $V$. However, even in this situation we completely lack a conceptual way to prove that $f$ is mirror to $V$.

Unfortunately, there exist polytopes $Q$ which have facets without Minkowski decompositions into $A$-triangles. So, at the moment, it is not clear what sort of combinatorial input we should consider on $Q$ in the general case.

Another approach to construct smoothings of the toric Fano variety $X_Q$ is pursued by Coates, Kasprzyk, and Prince \cite{laurent_inversion, thomas_cracked}; they embed $X_Q$ into a bigger toric variety $Z$ and try to deform it inside $Z$. This works very well in many explicit examples, but a general framework has yet to be discovered.

Finally, it is worth mentioning that Conjecture~\ref{conj} can be stated in all dimensions. Therefore, this might be a way to classify smooth Fano varieties that admit a toric degeneration.

\subsection*{Notation and conventions}

In a real vector space of finite dimension, a polyhedron is the intersection of finitely many closed half-spaces and a polytope is a compact polyhedron; equivalently, a polytope is the convex hull of a finite set.
We denote by $\conv{\cdot}$ the convex hull of a set.

All varieties and schemes are defined over $\CC$. We always use the following notation.
\begin{longtable}{lp{0.85\textwidth}}
$\dP_6$ & the smooth del Pezzo surface of degree 6 \\
$X$ & the projective cone over the anticanonical embedding of $\dP_6$ \\
$U$ & the affine cone over the anticanonical embedding of $\dP_6$ \\
$X_2$ & a general effective divisor of type $(1,1)$ in $\PP^2 \times \PP^2$ \\
$X_3$ & $\PP^1 \times \PP^1 \times \PP^1$ \\
$F$ & the lattice polygon associated to $\dP_6$ (see \eqref{eq:hexagon} and the left part of Figure~\ref{fig:Minkow}) \\
$P$ & the lattice polytope associated to $X$ (see \eqref{eq:pyramid} and Figure~\ref{fig:piramide}) \\
$f_a$ & $z \left(a + x+xy+y+ x^{-1}  + x^{-1} y^{-1} + y^{-1} \right) + z^{-1}$, for each $a \in \CC$ \\
$Q$ & an arbitrary Fano polytope (see Definition~\ref{def:Fano_polytope}) \\
$X_Q$ & the Fano toric variety associated to the Fano polytope $Q$
\end{longtable}

\subsection*{Acknowledgements} 

I am indebted to Tom Coates, Alessio Corti, Paul Hacking, Alexander Kasprzyk, Thomas Prince, and the other members of the Fanosearch group for countless fruitful conversations about the topics of this note. I wish to thank the organisers of the conference ``Birational geometry and moduli spaces'', held in Rome in June 2018, for giving me the opportunity to present a poster about this subject. Finally, I would like to thank Enrica Floris, Luigi Lunardon, and Diletta Martinelli for useful comments on a preliminary version of this note.

The author was funded by Kasprzyk's EPSRC Fellowship EP/N022513/1.

\section{The geometry of $X$}

\subsection{Toric geometry}
\label{sec:toric_geometry}

We now recall the basics of Fano toric varieties. We refer the reader to \cite[\S8.3]{cox_toric_varieties}, \cite[p.~25]{fulton_toric_varieties}, and \cite{fano_polytopes}.

\begin{definition} \label{def:Fano_polytope}
Let $N$ be a lattice of rank $n$. A \emph{Fano polytope} in $N$ is an $n$-dimensional polytope $Q \subseteq N_\RR$ such that  the origin $0 \in N$ lies in the interior of $Q$ and every vertex of $Q$ is a primitive lattice point of $N$.

The \emph{spanning fan} of a Fano polytope $Q$ in $N$ is the complete fan whose cones are the cones over the proper faces of $Q$.
We denote by $X_Q$ the toric variety associated to the spanning fan of a Fano polytope $Q$.
\end{definition}

For brevity, we say that $X_Q$ is associated to $Q$, and conversely.
If $Q$ is a Fano polytope of dimension $n$, then $X_Q$ is an $n$-dimensional complete toric variety which is Fano, i.e. its anticanonical divisor is $\QQ$-Cartier and ample. Every Fano toric variety arises in this way from a Fano polytope.

For example, consider the hexagon
\begin{equation} \label{eq:hexagon}
F = \conv{
\vectortwo{1}{0},
\vectortwo{1}{1},
\vectortwo{0}{1},
\vectortwo{-1}{0},
\vectortwo{-1}{-1},
\vectortwo{0}{-1}
} \subseteq \RR^2,
\end{equation}
which is depicted on the left of Figure~\ref{fig:Minkow}. It is clear that $F$ is a Fano polytope in $\ZZ^2$.
The toric variety associated to its spanning fan is the smooth del Pezzo surface of degree $6$, denoted by $\dP_6$. The anticanonical map of $\dP_6$ is a closed embedding into $\PP^6$.

Now imagine to put the hexagon $F$ into the plane $\RR^2 \times \{1 \}$ in $\RR^3$ and create the pyramid over it with apex at the point $(0,0,-1)$: this is the polytope 
\begin{equation} \label{eq:pyramid}
P = \conv{\vector{1}{0}{1}, \vector{1}{1}{1}, \vector{0}{1}{1}, \vector{-1}{0}{1}, \vector{-1}{-1}{1}, \vector{0}{-1}{1}, \vector{0}{0}{-1}} \subseteq \RR^3
\end{equation}
and is depicted in Figure~\ref{fig:piramide}. It is clear that $P$ is a Fano polytope in $\ZZ^3$.
Let $X$ be the toric variety associated to the spanning fan of $P$.
Let $U$ be the affine  toric open subscheme of $X$ associated to the hexagonal face of $P$, i.e. $U$ is the affine toric variety associated to the cone $\RR_{\geq 0} (F \times \{ 1 \})$. Hence $X$ (resp. $U$) is the projective (resp. affine) cone over the anticanonical embedding of $\dP_6$. We have that $X$ is a Fano threefold with an isolated non-terminal  canonical Gorenstein singularity at the vertex of the cone.

\subsection{Equations} \label{sec:equations}

The equations of the three closed embeddings $\dP_6 \subseteq \PP^6$, $U \subseteq \AA^7$ and $X \subseteq \PP^7$ are the same and can be conveniently described  in two ways.
Here, $x_1, \dots, x_7$ denote the homogeneous coordinates of $\PP^6$, the affine coordinates of $\AA^7$ and the last homogeneous coordinates of $\PP^7$, as $x_0$ is the remaining homogeneous coordinate of $\PP^7$.

The first way is:
\begin{equation*}
\rank \begin{pmatrix}
x_7 & x_1 & x_2 \\
x_4 & x_7 & x_3 \\
x_5 & x_6 & x_7
\end{pmatrix} \leq 1.
\end{equation*}
Note the repetition of $x_7$ on the diagonal. If two of the $x_7$'s had been two extra variables, these would have been the equations of the Segre embedding of $\PP^2 \times \PP^2$ in $\PP^8$. This shows that $X$ is the intersection of the projective cone over the Segre embedding of $\PP^2 \times \PP^2$ with two hyperplanes of $\PP^9$ passing through the vertex.

Now consider the cube
\begin{equation*}
\begin{tikzpicture}[%
  back line/.style={densely dotted},
  cross line/.style={preaction={draw=white, -,line width=6pt}}]
  \node (A) {$x_3$};
  \node [right of=A] (B) {$x_2$};
  \node [below of=A] (C) {$x_7$};
  \node [right of=C] (D) {$x_1$};
 
  \node (A1) [right of=A, above of=A, node distance=1cm] {$x_4$};
  \node [right of=A1] (B1) {$x_7$};
  \node [below of=A1] (C1) {$x_5$};
  \node [right of=C1] (D1) {$x_6$};
 
  \draw[back line] (D1) -- (C1) -- (A1);
  \draw[back line] (C) -- (C1);
  \draw[cross line] (D1) -- (B1) -- (A1) -- (A)  -- (B) -- (D) -- (C) -- (A);
  \draw (D) -- (D1) -- (B1) -- (B);
\end{tikzpicture}
\end{equation*}
where at the vertices there are the variables $x_1, \dots, x_7$. Note the repetition of $x_7$. The second way to describe the equations is to consider the determinants of all rectangles which can be formed with edges of the cube or with diagonals of faces of the cube. If one of the $x_7$'s had been an extra variable, these would have been the equations of the Segre embedding of $\PP^1 \times \PP^1 \times \PP^1$ into $\PP^7$. This shows that $X$ is the intersection of the projective cone over the Segre embedding of $\PP^1 \times \PP^1 \times \PP^1$ with a hyperplane of $\PP^8$ passing through the vertex.

The equations above also appear in \cite[Example~3.3]{priska_radloff}. Moreover, these two ways of describing the equations of $X$ in $\PP^7$ are called Tom and Jerry, respectively, in \cite{tom_jerry}.

\subsection{Minkowski sums and deformations of $U$} \label{sec:minkowski_deformations_U}

We first define the notion of Minkowski sum of polyhedra (for instance see \cite[\S1.1]{ziegler}). If $\Pi_1, \dots, \Pi_r$ are polyhedra in a real vector space, we define their \emph{Minkowski sum} to be the polyhedron
\[
\Pi_1 + \cdots + \Pi_r := \{ p_1 + \cdots + p_r \mid p_1 \in \Pi_1, \dots, p_r \in \Pi_r \}.
\]
When we have $\Pi = \Pi_1 + \cdots + \Pi_r $, we say that we have a \emph{Minkowski decomposition} of the polyhedron $\Pi$. We consider Minkowski decompositions up to translation: for instance, we consider the Minkowski decomposition $(p+ \Pi_1) + (-p + \Pi_2)$ to be equivalent to $\Pi_1 + \Pi_2$ for every vector $p$. Moreover, in what follows we require that the summands $\Pi_j$ are lattice polyhedra, i.e. their vertices belong to a fixed lattice.

The hexagon $F$ has two maximal Minkowski decompositions (see Figure~\ref{fig:Minkow}): one into 3 unitary segments
\begin{equation} \label{eq:minkowski_dec_segments}
F = \conv{\vectortwo{0}{0}, \vectortwo{1}{0}} 
+ \conv{\vectortwo{0}{0}, \vectortwo{0}{1}}
+ \conv{\vectortwo{0}{0}, \vectortwo{-1}{-1}} 
\end{equation}
and one into 2 triangles
\begin{equation}
\label{eq:minkowski_dec_triangles}
F = \conv{\vectortwo{0}{0}, \vectortwo{-1}{0}, \vectortwo{-1}{-1}  }
+ \conv{ \vectortwo{0}{0} , \vectortwo{1}{0}, \vectortwo{1}{1}}. 
\end{equation}

\begin{figure}[t]
\centering
\def\svgwidth{10cm}
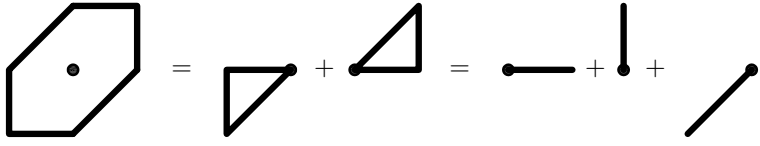
\caption{The two maximal Minkowski decompositions of the hexagon $F$.}
\label{fig:Minkow}
\end{figure}

Altmann \cite{altmann_minkowski_sums} has noticed that Minkowski decompositions of polytopes induce deformations of affine toric varieties (see also \cite{mavlyutov} and \cite{petracci_mavlyutov}). 
More precisely, from a Minkowski decomposition of a polytope $\Pi$ it is possible to construct an unobstructed deformation of the affine toric variety associated to the cone $\RR_{\geq 0} (\Pi \times \{ 1 \})$.

In the case at hand, the Minkowski decomposition \eqref{eq:minkowski_dec_segments} induces the deformation of $U$ over $\Spec \CC [u,v]$ given by the equations
\begin{equation*}
\rank \begin{pmatrix}
x_7 & x_1 & x_2 \\
x_4 & x_7 + u & x_3 \\
x_5 & x_6 & x_7 + v
\end{pmatrix} \leq 1.
\end{equation*}
The Minkowski decomposition \eqref{eq:minkowski_dec_triangles} induces the deformation of $U$ over $\Spec \CC[s]$ given by the equations obtained by taking minors of rectangles on edges and diagonals of faces of the following cube.
\begin{equation*}
\begin{tikzpicture}[%
  back line/.style={densely dotted},
  cross line/.style={preaction={draw=white, -,line width=6pt}}]
  \node (A) {$x_3$};
  \node [right of=A] (B) {$x_2$};
  \node [below of=A] (C) {$x_7$};
  \node [right of=C] (D) {$x_1$};
 
  \node (A1) [right of=A, above of=A, node distance=1cm] {$x_4$};
  \node [right of=A1] (B1) {$x_7 + s $};
  \node [below of=A1] (C1) {$x_5$};
  \node [right of=C1] (D1) {$x_6$};
 
  \draw[back line] (D1) -- (C1) -- (A1);
  \draw[back line] (C) -- (C1);
  \draw[cross line] (D1) -- (B1) -- (A1) -- (A)  -- (B) -- (D) -- (C) -- (A);
  \draw (D) -- (D1) -- (B1) -- (B);
\end{tikzpicture}
\end{equation*}

Moreover, Altmann \cite{altmann_versal_deformation} shows that the miniversal deformation of $U$ is (the completion of) the union of these two deformations and its base is $\CC \pow{s,u,v} / (su,sv)$.

\subsection{The two smoothings of $X$}
\label{sec:two_smoothings_of_X}

Now we want to study deformations of $X$.

\begin{proposition} \label{prop:deformations_X}
The base of the miniversal deformation of $X$ is $\CC \pow{s,u,v}/(su, sv)$ and has two irreducible components. 
The 2-dimensional component $(s=0)$ is associated to the Minkowski decomposition \eqref{eq:minkowski_dec_segments} and deforms $X$ to a general divisor $X_2 \in \vert \cO_{\PP^2 \times \PP^2}(1,1) \vert$.
The 1-dimensional component $(u=v=0)$  is associated to the Minkowski decomposition \eqref{eq:minkowski_dec_triangles} and deforms $X$ to $X_3 = \PP^1 \times \PP^1 \times \PP^1$.
\end{proposition}

\begin{proof}
Consider the local-to-global spectral sequence for $\Ext^\bullet_X(\Omega_X, \cO_X)$: the second page is $E^{p,q}_2 = \rH^q (X, \cExt^p_X(\Omega_X, \cO_X))$.
As $X$ has an isolated singularity, for all $p \geq 1$, the sheaf $\cExt^p_X(\Omega_X, \cO_X)$ is supported on the singular point of $X$; therefore, for all $p \geq 1$ and $q \geq 1$, $E^{p,q}_2 = 0$. 

Let $j \colon W \into X$ be the smooth locus. The sheaves $\cHom_X(\Omega_X, \cO_X)$
and $j_* \Omega_W^2 \otimes \cO_X(-K_X)$ are the same, because they are both reflexive and coincide on $W$.
As $X$ is toric and $-K_X$ is ample, by Bott--Steenbrink--Danilov vanishing \cite[Theorem~9.3.1]{cox_toric_varieties} (see also \cite{buch_thomsen, fujino_vanishing, mustata_vanishing}) one has $E_2^{0,q} = \rH^q (X, \cHom_X(\Omega_X, \cO_X)) = 0$ for all $q \geq 1$. This argument comes from the proof of \cite[Theorem~5.1]{totaro_jumping}.

Therefore $E_2$ is zero outside the line $q = 0$. This implies that, for all $p \geq 0$, the natural map
\[
\Ext^p_X (\Omega_X, \cO_X) \to \rH^0(X, \cExt^p_X (\Omega_X, \cO_X))
\]
is an isomorphism. Since the unique singular point of $X$ is contained in $U$ and $U$ is affine, we have $\rH^0(X, \cExt^p_X (\Omega_X, \cO_X)) = \Ext^p_U (\Omega_U, \cO_U)$ for all $p \geq 1$. This implies that, for all $p \geq 1$, the natural map
\begin{equation*}
\phi_p \colon \Ext^p_X (\Omega_X, \cO_X) \to \Ext^p_U (\Omega_U, \cO_U),
\end{equation*}
is an isomorphism.

We now consider the functors of infinitesimal deformations of $X$ and $U$: $\Def_X$ and $\Def_U$, which are covariant functors from the category of local finite $\CC$-algebras to the category of sets (see \cite[\S 3]{manetti_seattle}). There is an obvious map $\phi \colon \Def_X \to \Def_U$, which restricts a deformation of $X$ to $U$. Since $X$ is normal, $\Ext^1_X(\Omega_X, \cO_X)$ is the tangent space of $\Def_X$ and $\Ext^2_X(\Omega_X, \cO_X)$ is an obstruction space for $\Def_X$, and a similar statement holds for $U$. Since $\phi_1$ is bijective and $\phi_2$ is injective, we have that $\phi$ induces an isomorphism on tangent spaces and an injection on obstruction spaces. Therefore, by \cite[Remark~4.12]{manetti_seattle}, $\phi$ is smooth and induces an isomorphism on tangent spaces. In particular, the two functors $\Def_X$ and $\Def_U$ have the same hull, i.e. the bases of the miniversal deformations of $X$ and $U$ are the same.

The equations of the two deformations of $U$, given in \S\ref{sec:minkowski_deformations_U}, can be projectivised to construct deformations of $X$: it is enough to replace $s$, $u$ and $v$ by $s x_0$, $u x_0$ and $v x_0$. These are the two components of the miniversal deformation of $X$. The fact that they are associated to the two Minkowski decompositions \eqref{eq:minkowski_dec_segments} and \eqref{eq:minkowski_dec_triangles} of the hexagon $F$ follows from the discussion in \S\ref{sec:minkowski_deformations_U}.

From \S\ref{sec:equations} we know that $X$ is the intersection of the projective cone over the Segre embedding of $\PP^2 \times \PP^2$ with two hyperplanes of $\PP^9$ passing through the vertex of the cone. On the component $(s=0)$, in the deformation we are moving these two hyperplanes away from the vertex. Therefore, the general fibre over this component is $X_2$, a general $(1,1)$-divisor in $\PP^2 \times \PP^2$.

Recall that $X$ is the intersection of the projective cone over the Segre embedding of $\PP^1 \times \PP^1 \times \PP^1$ with a hyperplane of $\PP^8$ passing through the vertex. On the component $(u=v=0)$, in the deformation we are moving this hyperplane of $\PP^8$ away from the vertex. Therefore, the general fibre on this component is $X_3 = \PP^1 \times \PP^1 \times \PP^1$.
\end{proof}

\section{Mirror Symmetry}

\subsection{Gromov--Witten invariants and quantum periods} \label{sec:quantum_periods}

The quantum period of a smooth Fano variety is a generating function for some genus zero Gromov--Witten invariants. The regularised quantum period is a slightly different version, which is convenient for our description of Mirror Symmetry.

\begin{definition}[\cite{mirror_symmetry_and_fano_manifolds, quantum_periods_3folds, przyjalkowski_landau_ginzburg_fano}] \label{def:quantum_period}
The \emph{quantum period} and the \emph{regularised quantum period} of a smooth Fano variety $V$ are the following power series 
\begin{gather*}
G_V(t) = 1 + \sum_{\beta \in \rH_2(V,\ZZ)}  \langle [\mathrm{pt}] \psi^{-K_V \cdot \beta -2}\rangle_{0,1,\beta}^V t^{-K_V \cdot \beta} \in \QQ \pow{t} \\
\widehat{G}_V(t) = 1 + \sum_{\beta \in \rH_2(V,\ZZ)} (-K_V \cdot \beta)!  \langle [\mathrm{pt}] \psi^{-K_V \cdot \beta -2}\rangle_{0,1,\beta}^V t^{-K_V \cdot \beta} \in \QQ \pow{t}
\end{gather*}
where $\langle [\mathrm{pt}] \psi^{-K_V \cdot \beta -2}\rangle_{0,1,\beta}^V$ denotes the $1$-marked genus zero Gromov--Witten invariant of curve class $\beta$ associated to the cohomology class of a point in $V$ and gravitational descendant of order $-K_V \cdot \beta - 2$.
\end{definition}

Roughly speaking, $\langle [\mathrm{pt}] \psi^{-K_V \cdot \beta -2}\rangle_{0,1,\beta}^V$ is the number of rational curves in $V$ of class $\beta$ passing through a fixed general point of $V$ and satisfying a certain condition on their complex structure. Therefore, the quantum period $G_V$ gives information about rational curves in $V$. The series $G_V$ is a symplectic invariant of $V$, so it does not change if $V$ is deformed to another smooth Fano variety through a deformation with smooth fibres.

If the anticanonical line bundle $\cO_V(-K_V)$ is divisible by a positive integer $m$ inside the Picard group of the smooth Fano variety $V$, then only powers of $t^m$ appear in the (regularised) quantum period of $V$.

It is also possible to define quantum periods for Fano varieties with quotient singularities \cite[\S3.3]{oneto_petracci}.

It is known how to compute the quantum period of smooth Fano varieties which are either toric or complete intersections in smooth Fano toric varieties \cite{quantum_lefschetz_tom_givental, givental_equivariant}. 
The quantum periods of all smooth Fano varieties of dimension $\leq 3$ have been computed by Coates, Corti, Galkin, and Kasprzyk \cite{quantum_periods_3folds}. In particular, we have the following formulae for $X_2$ and $X_3$.

\begin{proposition}[{\cite{quantum_periods_3folds}}] \label{prop:quantum_periods_X2_X3}
The quantum periods and the regularised quantum periods of $X_2$ and $X_3$ are the following.
\begin{gather*}
G_{X_2} (t) = \sum_{l=0}^\infty \sum_{m=0}^\infty \frac{(l+m)!}{(l!)^3 (m!)^3} t^{2l + 2m} \\
G_{X_3}(t) = \sum_{l=0}^\infty \sum_{m=0}^\infty \sum_{n=0}^\infty \frac{1}{(l!)^2 (m!)^2 (n!)^2} t^{2l + 2m + 2n}
\end{gather*}
\begin{gather*}
\widehat{G}_{X_2} (t) = 1 + 4t^2 + 60 t^4 + 1120 t^6 + 24220 t^8 + 567504 t^{10} + \cdots \\
\widehat{G}_{X_3} (t) = 1 + 6t^2 + 90 t^4 + 1860 t^6 + 44730 t^8 + 1172556 t^{10} + \cdots 
\end{gather*}
\end{proposition}

\subsection{Laurent polynomials} \label{sec:laurent_polynomials}

Let $\CC[x_1^\pm, \dots, x_n^\pm]$ be the ring of Laurent polynomials in $n$ variables with coefficients in $\CC$. To every monomial $x^{\boldsymbol{i}} = x_1^{i_1} \cdots x_n^{i_n}$ we associate the point $\boldsymbol{i} = (i_1,\dots,i_n) \in \ZZ^n$. The \emph{Newton polytope} of a Laurent polynomial $f$ is the convex hull of the lattice points that correspond to the monomials that appear in $f$, i.e. if 
$f = \sum_{\boldsymbol{i} \in \ZZ^n} a_{\boldsymbol{i}} x^{\boldsymbol{i}}$
 then
\[
\mathrm{Newt}(f) = \conv{ \boldsymbol{i} \in \ZZ^n \mid a_{\boldsymbol{i}} \neq 0 } \subseteq \RR^n.
\]
If $Q$ is a lattice polytope in $\ZZ^n$, we say that a Laurent polynomial $f \in \CC[x_1^\pm, \dots, x_n^\pm]$ is \emph{supported} on $Q$ if every monomial appearing in $f$ corresponds to a lattice point of $Q$, or equivalently if $\mathrm{Newt}(f) \subseteq Q$.

Given a Fano polytope $Q$ in $\ZZ^n$, Kasprzyk and Tveiten \cite{al_ketil} have introduced and studied a particular class of Laurent polynomials supported on $Q$; they call them \emph{maximally mutable}, because these behave well with respect to mutations of Fano polytopes \cite{sigma}. The definition of maximally mutable Laurent polynomials in dimension 2 can be also found in \cite{del_pezzo_surfaces}.
We are not going to define maximally mutable Laurent polynomials here, we just mention some properties in a particular case.

In dimension 3, when the Fano toric threefold $X_Q$ has Gorenstein singularities (equivalently $Q$ is a reflexive polytope of dimension 3), every maximally mutable Laurent polynomial on $Q$ is such that:
\begin{itemize}
\item the coefficient of the monomial $1$, corresponding to the origin of $\ZZ^3$, is $0$;
\item the monomials corresponding to the vertices of $Q$ have coefficients equal to $1$;
\item on the edges of $Q$ there are binomial coefficients. (For example, the 4 lattice points of an edge with lattice length $3$ have coefficients $1,3,3,1$.)
\end{itemize}

In the case of the polytope $P$, a Laurent polynomial is supported on $P$ if and only if its monomials are among $1$, $xz$, $xyz$, $yz$, $x^{-1}z$, $x^{-1} y^{-1} z$, $y^{-1} z$, $z^{-1}$, which correspond to the lattice points of $P$. The Laurent polynomials on $P$ which satisfy the three properties above form a 1-dimensional family
\begin{equation*}
f_a = z \left( x+xy+y+\frac{1}{x} + \frac{1}{xy} + \frac{1}{y} + a \right) + \frac{1}{z}
\end{equation*}
with parameter $a \in \CC$. Here $a$ is the coefficient of the centre of the hexagonal facet of $P$. Kasprzyk and Tveiten \cite{al_ketil} show that there are exactly two maximally mutable Laurent polynomials on $P$, namely $f_a$ with $a=2$ and $a=3$. One notices that, in these two cases, the restriction of $f_a$ to the hexagonal facet of $P$ is reducible:
\begin{gather*}
f_2 = z(1+x)(1+y)(1+x^{-1}y^{-1}) + z^{-1}, \\
f_3 = z(1+x^{-1}y^{-1} + y^{-1})(1 +xy +y) + z^{-1}.
\end{gather*}
The Newton polytopes of the three factors of $(1+x)(1+y)(1+x^{-1}y^{-1})$ are the three unitary segments appearing in the Minkowski decomposition \eqref{eq:minkowski_dec_segments} of the hexagon $F$. The Newton polytopes of the factors of $(1+x^{-1}y^{-1} + y^{-1})(1 +xy +y)$ are the two triangles appearing in the Minkowski decomposition \eqref{eq:minkowski_dec_triangles} of the hexagon $F$. 
The Laurent polynomials $f_2$ and $f_3$ are \emph{Minkowski polynomials} in the sense of \cite{sigma}.

\subsection{Classical periods} \label{sec:classical_periods}

We now define the classical period of a Laurent polynomial in $n$ variables.

\begin{definition}[\cite{sigma, galkin_usnich}] \label{def:classical_period}
The \emph{classical period} of $f \in \CC[x_1^\pm, \dots, x_n^\pm]$ is the power series
\begin{align*}
\pi_f(t) &= \left( \frac{1}{2 \pi \rmi} \right)^n \int_{\Gamma_\varepsilon} \frac{1}{1 - t f(x_1, \dots, x_n)} \frac{\rmd x_1 \wedge \cdots \wedge \rmd x_n}{x_1 \cdots x_n} \\
 &= \sum_{k = 0}^\infty \mathrm{coeff}_1 (f^k) t^k
\end{align*}
where in the first formula we are integrating a holomorphic $n$-form of the torus $(\CC^\times)^n = \Spec \CC[x_1^\pm, \dots, x_n^\pm]$ over the real torus $\Gamma_\varepsilon = \{ \vert x_1 \vert = \cdots = \vert x_n \vert = \varepsilon \} \subseteq (\CC^\times)^n$, for some $0 < \varepsilon \ll 1$, and $\mathrm{coeff}_1(f^k) \in \CC$ is the coefficient of the monomial $1 = x_1^0 \cdots x_n^0$ in the Laurent polynomial $f^k$.
\end{definition}

The equality between the two formulae in the definition above comes from applying Cauchy's integral formula $n$ times. The classical period $\pi_f$ is related to the Hodge theory of the fibres of $f \colon (\CC^\times)^n \to \AA^1$.

One can see that the classical period of the Laurent polynomial $f_a$ is
\begin{align*}
\pi_{f_a}(t) &= 1 + 2at^2 +
(6a^2 + 36)t^4 +
(20a^3 + 360a + 240) t^6  \\
&+ (70a^4 + 2520a^2 + 3360a + 6300) t^8 \\
&+ (252a^5 + 15120a^3 + 30240a^2 + 113400a + 90720) t^{10} + \cdots
\end{align*}
for every $a \in \CC$. In particular,
\begin{gather*}
\pi_{f_2} (t) = 1 + 4t^2 + 60 t^4 + 1120 t^6 + 24220 t^8 + 567504 t^{10} + \cdots, \\
\pi_{f_3} (t) = 1 + 6t^2 + 90 t^4 + 1860 t^6 + 44730 t^8 + 1172556 t^{10} + \cdots .
\end{gather*}

\subsection{Fano/Landau--Ginzburg correspondence} \label{sec:Fano-Landau_Ginzburg}

Mirror Symmetry \cite{mirror_symmetry_and_fano_manifolds, katzarkov_victor} predicts that the mirror of a smooth Fano $n$-fold $V$ is a pair $(Y,w)$, called \emph{Landau--Ginzburg model}, where $Y$ is an $n$-fold and  $w \in \Gamma(Y, \cO_Y)$ is a regular function. The Gromov--Witten theory of $V$ should be related to the Hodge theory of the fibres of $w \colon Y \to \AA^1$ as follows: the regularised quantum period $\widehat{G}_V$ (see Definition~\ref{def:quantum_period}) of $V$ coincides with the period $\pi_w$ which is defined as
\begin{equation} \label{eq:hodge_period}
\pi_w(t) = \int_{\Gamma} \frac{\Omega}{1-tw}
\end{equation}
where $\Omega$ is an appropriate holomorphic $n$-form on $Y$ and $\Gamma \in \rH_n(Y;\ZZ)$ is such that $\int_\Gamma \Omega = 1$.

Under some circumstances (which conjecturally and experimentally should coincide with when there is a toric degeneration of $V$) there is an open subset of $Y$ that is isomorphic to the torus $(\CC^\times)^n = \Spec \CC[x_1^\pm, \dots, x_n^\pm]$. In this case the restriction of $w$ to this open subset gives a Laurent polynomial $f \in \CC[x_1^\pm, \dots, x_n^\pm]$. In this situation the period $\pi_w$ in \eqref{eq:hodge_period}, when $Y = (\CC^\times)^n$, $\Gamma = \{ \vert x_1 \vert = \cdots = \vert x_n \vert = \varepsilon \}$ and $\Omega = (2 \pi \rmi)^{-n} (x_1 \cdots x_n)^{-1} \rmd x_1 \cdots \rmd x_n$, becomes the classical period of a Laurent polynomial (see Definition~\ref{def:classical_period}).

Thus, a down-to-earth formulation of Mirror Symmetry between smooth Fano varieties and Laurent polynomials is the following.

\begin{definition}[{\cite{mirror_symmetry_and_fano_manifolds,
przyjalkowski_landau_ginzburg_fano, victor_lg}}]
\label{def:mirror}
A Laurent polynomial $f \in \CC[x_1^\pm, \dots, x_n^\pm]$ is  \emph{mirror} to a smooth Fano variety $V$ of dimension $n$ if the classical period of the former coincides with the regularised quantum period of the latter: $\widehat{G}_V = \pi_f$.
\end{definition}

The equality $\widehat{G}_V = \pi_f$ is equivalent to $G_V$ being equal to the oscillatory integral
\[
\left( \frac{1}{2 \pi \rmi} \right)^n \int_{\Gamma_\varepsilon} \mathrm{e}^{tf} \frac{\rmd x_1 \wedge \cdots \wedge \rmd x_n}{x_1 \cdots x_n} = \sum_{k = 0}^\infty \frac{\mathrm{coeff}_1 (f^k)}{k!} t^k.
\]
Moreover, the equality $\widehat{G}_V = \pi_f$ can be upgraded to an equality between the Gauss--Manin connection on the middle cohomology of the fibres of $f$ and the Dubrovin connection of the quantum D-module of $V$ (see \cite{golyshev_classification}).

\begin{proposition}[{\cite{sigma,quantum_periods_3folds}}] \label{prop:f_23_mirror_X_23}
The Laurent polynomial $f_2$ (resp. $f_3$) is mirror to the smooth Fano threefold $X_2$ (resp. $X_3$).
\end{proposition}

\begin{proof}
Set $a=2$ or $a=3$. We need to show that the two power series $\pi_{f_a}$ and $\widehat{G}_{X_a}$ coincide.
By comparing the formulae given in Proposition~\ref{prop:quantum_periods_X2_X3} and at the end of \S\ref{sec:classical_periods}, one can check the equality of finitely many coefficients.
In order to prove the equality of all coefficients, one has to show that $\pi_{f_a}$ and $\widehat{G}_{X_a}$ satisfy the same linear differential equation; this is done in \cite{sigma} and \cite{quantum_periods_3folds}.
\end{proof}

Now we are ready to illustrate Conjecture~\ref{conj} in the example of the projective cone over the smooth del Pezzo surface of degree 6, which is the running example of this note.

In Proposition~\ref{prop:deformations_X} we saw that the Minkowski decomposition \eqref{eq:minkowski_dec_segments} of the hexagonal facet of $P$ into three unitary segments is associated to the smoothing of $X$ to $X_2$.
In \S\ref{sec:laurent_polynomials} we saw that the restriction of $f_2$ to the hexagonal facet of $P$ is reducible and that the Newton polytopes of its three factors are the three unitary segments appearing in the Minkowski decomposition \eqref{eq:minkowski_dec_segments}. By Proposition~\ref{prop:f_23_mirror_X_23} we know that $f_2$ is mirror to $X_2$. This is an instance of Conjecture~\ref{conj}: from the combinatorial input of the Minkowski decomposition of the hexagonal facet of $P$ into three unitary segments we have constructed the smoothing $X_2$ of $X$ and the Laurent polynomial $f_2$ which is mirror to $X_2$.

In a completely analogous manner, we can observe that the Minkowski decomposition \eqref{eq:minkowski_dec_triangles} of the hexagonal facet of $P$ into two triangles induces the smoothing $X_3$ of $X$ and the Laurent polynomial $f_3$. This provides another example for Conjecture~\ref{conj} because $f_3$ is mirror to $X_3$.

As mentioned in \S\ref{sec:general}, given a reflexive polytope $Q$ of dimension 3, from the combinatorial datum given by the choice of a Minkowski decomposition of each facet of $Q$ into $A$-triangles, one constructs an associated Laurent polynomial $f$ supported on $Q$. From the same combinatorial datum on $Q$ (with a slight additional condition which we do not mention here), by \cite{chp} it is possible to construct a smoothing $V$ of the toric Fano threefold $X_Q$ associated to $Q$. It is conjectured that the smooth Fano threefold $V$ is mirror to the Laurent polynomial $f$.

This circle of ideas should be considered as an approach to the problem of classifying smooth Fano varieties of dimension $\geq 4$. Indeed, computers can classify Fano polytopes; therefore, once one has developed a combinatorial technology for smoothing toric Fano varieties, one should be able to construct all smooth Fano varieties which admit a toric degeneration.

There is another difficulty: a smooth Fano variety may have many toric degenerations, hence may arise from several polytopes.
For instance, $X_3 = \PP^1 \times \PP^1 \times \PP^1$ is itself toric and degenerates to the toric Fano $X$. Conjecturally, these many toric degenerations of a smooth Fano correspond to many mirror Laurent polynomials; these Laurent polynomials are related via certain birational transformations of the torus $(\CC^\times)^n$, which are called \emph{mutations} \cite{sigma, cruz_morales_galkin, galkin_usnich} and preserve the classical periods. Therefore, it is conjectured that deformation families of smooth Fano varieties of dimension $n$ are in one-to-one correspondence with mutation-equivalence classes of some ``special'' Laurent polynomials in $n$ variables.
We are not going to expand on this here because otherwise it would lead us far beyond the scope of this note.

\bibliography{Biblio_Roma}


\end{document}